\documentclass[12pt]{article}

\usepackage[utf8]{inputenc}
\usepackage{amsmath, amssymb, amsthm}
\usepackage{graphicx}
\usepackage{hyperref}
\hypersetup{
colorlinks=true,
urlcolor=blue,
citecolor=blue}
\usepackage{setspace}
\usepackage[a4paper, margin=0.5in]{geometry} 
\newcommand{\bc}{\begin{center}}
\newcommand{\ec}{\end{center}}

\numberwithin{equation}{section}

\newtheorem{theorem}{Theorem}[section]

\newtheorem{corollary}[theorem]{Corollary}

\newtheorem{definition}[theorem]{Definition}

\newtheorem{proposition}[theorem]{Proposition}

\title{\small Using infinitesimal symmetries   for determining the  first Maxwell time of geometric control problem on SH(2)}
\author{
  \normalsize Soukaina Ezzeroual\textsuperscript{1}, Brahim Sadik\textsuperscript{1} \\[0.2cm]
  \small \textit{\textsuperscript{1}Department of Mathematics, Cadi Ayyad University, Faculty of Sciences Semlalia, Marrakesh, Morocco} \\[0.1cm]
  \small \textit{Email: s.ezzeroual.ced@uca.ac.ma, sadik@uca.ac.ma}
}
\date{}
\begin{document}
\maketitle
\begin{abstract}
In this work, we utilize infinitesimal symmetries to compute Maxwell points which play a crucial role in studying sub-Riemannian control problems.  By examining the infinitesimal symmetries of the geometric control problem on the $\text{SH(2)}$ group,  particularly through  its Lie algebraic structure, we identify invariant quantities and constraints that streamline the Maxwell point computation.  
\end{abstract}
\section{Introduction}
\label{Sec:1}

Describing locally optimal trajectories is crucial in dynamical systems, control theory, robotics, and other areas where precise motion control is necessary. In \cite{10}, several optimal control problems are investigated to determine optimal solutions and analyze the optimality of geodesics, specifically focusing on the first Maxwell time, which indicates how long a geodesic trajectory remains optimal. For more details on sub-riemannian geometry and geometric control theory, we refer the reader to \cite{1, 10}.

An optimal control problem concerns finding controls that steer the studied system from one state to another while minimizing certain quantities, often related to energy or time. A simple and interesting example is the  geometric control problem  on the Lie group $\text{SH}(2)$ which involves controlling a system whose state space is the special Euclidean group in two dimensions.  In \cite{11}, the author makes significant contributions to this problem by presenting optimal solutions derived using Pontryagin's Maximum Principle (PMP). Additionally, the author analyzes the discrete symmetries of this problem to understand the optimality of geodesics.

Infinitesimal symmetries are transformations that leave the fundamental equations of a system invariant. They  are also powerful tools in studying problems in theoretical physics and dynamic systems, providing insight into conserved quantities and invariant structures within a system. By studying the Lie algebra of these symmetries, we can often derive conserved quantities that serve   in the computation of some properties of some control systems.

In this note, we present a method that combines algebraic and geometric tools to uncover infinitesimal symmetries of a sub-riemannian control problem. Then we apply it to  the geometric control problem on the Lie group $\text{SH}(2)$. We compute the Lie algebra of these symmetries and use it to study the loss of optimality in trajectories and compute Maxwell points.

Sub-Riemannian geodesics are generally not optimal for all times. Each geodesic has a specific point where it ceases to be optimal. At this stage, the role of the infinitesimal symmetries of our problem on the group $\text{SH(2)}$ becomes crucial. Using the Lie algebra of symmetries, we identify transformations that map a given geodesic to another geodesic. We show that there is a subgroup of symmetries isomorphic to the group $\rm{SO(1,2)}$. Using this action, we determine  the set of Maxwell points, and subsequently, the first Maxwell time corresponding to our infinitesimal symmetries. All of this is detailed in Section \ref{sec:4} of our main results.

\section{Preliminaries on Geometric Control Theory}
\label{Sec:2}
In this section, we give some preliminaries on geometric control theory. For more details on the subject, see \cite{1}
\subsection{Controllability of control affine systems}
A control affine system on a manifold ${\rm M}$ is any differential system of the form
\begin{equation}
\label{eq1}
\frac{dx}{dt}=X_0(x)+\sum_{i=1}^{m}u_i(t)X_i(x), 
\end{equation}
 where 
 \begin{itemize}
 \item $x=x(t)$ is the state,
 \item $X_0,X_1,\ldots,X_m$ are smooth vector fields on $\rm M$,
 \item $u_1(t),\ldots,u_m(t)$ are the control imputs. 
\end{itemize}    
The vector field $X_0$ is called the drift of the control system and when it equals zero, we say that the system is driftless. The control system (\ref{eq1}) is said to be controllable if the following holds: Given two states $x_0$ and $x_1$ in ${\rm  M}$ there exists  a control $u(t)$ that
steers $x_0$ to $x_1$ in some (or a priori fixed) positive time $T$.   \\

Let  $\mathcal{F}$ be  a family of smooth vector fields on ${\rm M}$, then the Lie algebra generated by $\mathcal{F}$ is the smallest Lie algebra that contains $\mathcal{F}$. It is obtained by considering all linear combinations of elements of $\cal F$, taking all Lie brackets of these, considering all linear combinations of these, and continuing  so on. It will be denoted by ${\rm Lie}(\mathcal{F})$ and ts evaluation at any point $q\in {\rm M}$ will be denoted by $Lie_q(\mathcal{F})$.
The following result gives a necessary and sufficient condition for a driftless control affine system to be controllable.
\begin{proposition}
\label{pro1}(\cite{7}).
The control affine  system $\frac{d x}{d t}=\sum_{i=1}^m u_i X_i(x)$ with $u=\left(u_1, \ldots, u_m\right) \in \mathbb{R}^m$ is  controllable if and only if ${\rm Lie}_q \mathcal{F}=T_q {\rm M}$, for all $q \in {\rm M}$.\footnote{Here  $T_q{\rm M} $ denotes the tangent space to ${\rm M}$ at the point $q$}
\end{proposition}
\subsection{Sub-Riemannian Problem and optimal solutions} 
A sub-Riemannian manifold is a triplet $({\rm M}, \Delta, {\rm g})$ such that ${\rm M}$ is a manifold of dimension $n$, $\Delta$ is a smooth distribution of rank $k \leq n$, and ${\rm g}$ is a Riemannian metric defined on $\Delta$. More precisely, 
 $ \forall q\in {\rm M},\,\,   \Delta_q $ is a subspace of $T_q{\rm M}$ of dimension $k$ and 
${\rm g}$ is a family of inner products ${\rm g}_q$ in $ \Delta_q$. \\
We say that a Lipschitz curve $\gamma : [0,t_1] \mapsto {\rm M}$  is admissible if $\dot{\gamma}(t) \in \Delta_{\gamma(t)}$ for all $t \in [0,t_1]$. Then the sub-Riemannian length of such a curve is defined by:
$$
\ell(\gamma) = \int_0^{t_1} \sqrt{{\rm g}_{\gamma(t)}(\dot{\gamma}(t), \dot{\gamma}(t))} \, dt.
$$
 Given two points $q_0$ and $q_1$ of ${\rm M}$, the sub-Riemannian distance between $q_0$ and $q_1$ is stated by
$$
{\rm d}(q_0, q_1) = \inf \{ \ell(\gamma) \mid \gamma \text{ admissible}, \gamma(0) = q_0, \gamma(t_1) = q_1 \}.
$$
A sub-Riemannian problem is a control problem on a sub-Riemannian manifold $({\rm M}, \Delta, {\rm g})$ where one   seek for admissible curves $\gamma$ that satisfy the property $\ell(\gamma) = d(q(0), q(t_1))$.
Suppose there exists a family of smooth vector fields \( X_1, \ldots, X_k \) that forms an orthonormal frame on $(\Delta, {\rm g})$, i.e. $\forall q\in {\rm M},\,  \Delta_q = {\rm span}\{ X_1(q), \ldots, X_k(q)\}$ and ${\rm g}_q(X_i(q), X_j(q)) = \delta_{ij}.$
Thus, sub-Riemannian minimizers (or optimal solutions) are the solutions of the following optimal control problem on ${\rm M}$:
\begin{align}
\dot{x} &=\sum_{i=1}^k u_i X_i(x), \quad u=(u_1,...,u_k)\in \mathbb{R}^{k}, \label{eq2}\\
x(0) &=q_0, \quad x(t_1)=q_1, \label{eq3}\\
\ell&=\int_{0}^{t_1} (\sum_{i=1}^k u_i^{2} )^{1/2} dt \longrightarrow \min. \label{eq4} 
\end{align}
If we add the condition $\sum_{i=1}^k u_i^2 \leq 1$, the system  given by equations (\ref{eq2})-(\ref{eq4}), becomes: 
\begin{equation}
\label{time-min}
\begin{aligned}
& \dot{x}=\sum_{i=1}^k u_i X_i(x), \quad  \sum_{i=1}^k u_i^2 \leq 1, \\
& x(0)=q_0, \quad x(t_1)=q_1, \\
& t_1 \rightarrow \min .
\end{aligned}
\end{equation}
By Filipov's theorem (\cite{10}),  the existence of optimal solutions for the optimal control problem (\ref{time-min}) is guaranteed.\\
The Pontryagin Maximum Principle    \cite{1} provides necessary conditions for  optimality of trajectories solutions of   sub-Riemannian control problems. Following this principle, we first   compute trajectories, called extremals,  of a dynamic system in the cotangent bundle of the variety $\rm M$.  Then the projections of the extremals on the state space $\rm M$ constitute the optimal solutions and are called geodesics.  
A point $\gamma(t_1)$ is called a Maxwell point along a geodesic $\gamma$ if  there exists another geodesic $\widetilde{\gamma} \not \equiv \gamma$ such that $\gamma(t_1)=\widetilde{\gamma}(t_1)$. This means that  there exists a geodesic $\widehat{\gamma}$ coming to the point $q_1=\gamma\left(t_1\right)$ earlier than $\gamma$.
\section{Computing  infinitesimal symmetries of a sub-Riemannian problem}
Let \(({\rm M}, \Delta, {\rm g})\) be a sub-Riemannian manifold, where  $\rm M$ is endowed with a Lie group structure of dimension $n$ that is left-invariant. This means that 
\[
\Delta_{ab} = L_{a *} \Delta_b, \quad {\rm g}_b( v, w )= g_{ab}( L_{a *} v, L_{a*} w ), \quad \forall a, b \in M,
\]
where $L_{a *}$ denotes the differential of the left translation by $a$. We   consider  the sub-Riemannian system (\ref{eq2})-(\ref{eq3})-(\ref{eq4}) on  the Lie group \(({\rm M}, \Delta, {\rm g})\) and we assume it is  controllable. Then    ${\rm Lie}_q {\Delta} = T_q {\rm M}$ for every $q \in M$. Therefore 
\begin{equation}
\label{eq00}
{\rm Lie}_q (\Delta) = \operatorname{span}\{ X_1(q), \ldots, X_k(q), X_{k+1}(q), \ldots, X_n(q) \}
\end{equation}
for all $q \in M$, where $\Delta$ is given by the vector fields $X_1,\ldots,X_k$ and  $X_{k+1}, \ldots, X_n\in Lie(\Delta)$. Furthermore, suppose the problem admits optimal solutions. 
\begin{definition} (\cite{2}).
Let $M$ be a $2 m+1$ dimensional manifold. A sub-Riemannian structure on $\rm{M}$ is said to be contact if $\Delta$ is a contact distribution, i.e. $\Delta=\operatorname{ker} \omega$, where $\omega \in \Lambda^1 M$ satisfies $\left(\bigwedge^m d \omega\right) \wedge \omega \neq 0$.
\end{definition}
By reference to \cite{6}, which provides us with an important result regarding the conditions that a symmetry vector must satisfy, we have that: 
A vector \( v \) is a symmetry of our system if it satisfies the following two conditions: 
\begin{equation}
\label{eqa}
    \mathcal{L}_v(\Delta) \subset \Delta \quad \text{and} \quad \mathcal{L}_v(\rm{g}) = 0, \quad \text{where } \mathcal{L} \text{ denotes the Lie derivative.}
\end{equation}

In what follows, we outline the    necessary steps for computing  the symmetry algebra of a system defined by equations (\ref{eq2})-(\ref{eq3})-(\ref{eq4}). We consider $\Delta = \{ \Delta_q \subset T_q M \mid q \in M \}$, where $\Delta_q = \operatorname{span}\{ X_1(q), \ldots, X_k(q) \}$. Additionally, $\rm{g}= \sum_{i=1}^{n} a_i dq_i\otimes dq_i$, where $a_i$ are some constant,  represents a Riemannian metric such that $\rm{g}(X_i, X_j) = \delta_{ij}$. Furtheremore we suppose that $\Delta$ is a contact distribution, meaning that $\ker(\omega) = \Delta$, where $\omega \in \Lambda^1(\rm{M})$ is a 1-form on the manifold $\rm{M}$. In terms of the local coordinates $q_1, \dots, q_n$, the contact form can be written as $\omega = \sum_{i=1}^{n} e_i \, dq_i$, where $e_i$ are scalar functions of the coordinates $q_1, \dots, q_n$. An  arbitrary vector field $v$  can be expressed as  $ v = h_1(q_1, \ldots, q_n) X_1 + \ldots + h_n(q_1, \ldots, q_n) X_n$ and it is  a symmetry vector if it satisfies the two preceding conditions (\ref{eqa}).  We recall that   in local coordinates, each \(X_j\) is of the form 
$X_j = \sum_{i=1}^{n} g_i^j \partial q_i$, where $g_i^j$ are some scalar functions defined on $\rm{M}$.\\

Our approach begins by calculating $\mathcal{L}_v(X_1), \ldots, \mathcal{L}_v(X_k)$: 

\begin{equation}
\begin{aligned}
\mathcal{L}_v(X_j) &= [v, X_j] \\
                   &= [h_1 X_1 + \ldots + h_n X_n, X_j] \\
                   &= -[X_j, h_1 X_1 + \ldots + h_n X_n] \\
                   &= -h_1 [X_j, X_1] - (X_j h_1) X_1 - \ldots - h_j [X_j, X_j] - (X_j h_j) X_j \\
                   &\quad - \ldots - h_n [X_j, X_n] - (X_j h_n) X_n \\
                   &= -h_1 \sum_{i=1}^{n} s_i \frac{\partial}{\partial q_i} 
                   - \left( \sum_{i=1}^{n} g_i^j \frac{\partial h_1}{\partial q_i} \right) \sum_{i=1}^{n} g_i^1 \frac{\partial}{\partial q_i} \\
                   &\quad - \ldots 
                   - \left( \sum_{i=1}^{n} g_i^j \frac{\partial h_j}{\partial q_i} \right) \sum_{i=1}^{n} g_i^j \frac{\partial}{\partial q_i} \\
                   &\quad - \ldots 
                   - h_n \sum_{i=1}^{n} t_i \frac{\partial}{\partial q_i} 
                   - \left( \sum_{i=1}^{n} g_i^j \frac{\partial h_n}{\partial q_i} \right) \sum_{i=1}^{n} g_i^n \frac{\partial}{\partial q_i}
\end{aligned}
\end{equation}
where $s_i$ and $t_i$ are some scalar functions defined on $\rm{M}$. Then, applying the condition $\omega(\mathcal{L}_v(X_j)) = 0$ yields    an equation of  the form:
\begin{equation}
\label{con 1}
\sum_{i \neq j}^{n} b_i h_i + \sum_{i=1}^{n} c_i \frac{\partial h_1}{\partial q_i} + \ldots + \sum_{i=1}^{n} d_i \frac{\partial h_n}{\partial q_i} = 0,
\end{equation}
where  $b_i$, $c_i$ and $d_i$ are some scalar functions. Extending the   procedure to the all vector fields $X_1, \ldots, X_k$, this yields a system of order $k$, consisting of equations of the form (\ref{con 1}).\\
 Next, we compute the Lie derivative of $\rm{g}$. We have:
\begin{equation}
\begin{aligned}
\mathcal{L}_v(\rm{g}) &= \mathcal{L}_v\left(\sum_{i=1}^{n} a_i \, dq_i \otimes dq_i\right) \\
                 &= \sum_{i=1}^{n} a_i \, \mathcal{L}_v(dq_i \otimes dq_i) \\
                 &= \sum_{i=1}^{n} 2 a_i \, dq_i \, \mathcal{L}_v(dq_i) \\
                 &= \sum_{i=1}^{n} 2 a_i \, dq_i \, \left(i(v) \, d(dq_i) + d(i(v) \, dq_i)\right) \\
                 &= \sum_{i=1}^{n} 2 a_i \, dq_i \, d(i(v) \, dq_i) \\
                 &= \sum_{i=1}^{n} 2 a_i \, dq_i \, d(dq_i(v))
\end{aligned}
\end{equation}
where $i(v)$, is the interior product. Condition 
$\mathcal{L}_v(\rm{g})=0$ provides a system of order $\frac{n(n+1)}{2}$, consisting of equations of the form:
\begin{equation}
\label{con 2}
\sum_{i=1}^{n} k_i h_i + \sum_{i=1}^{n} l_i \frac{\partial h_1}{\partial q_i} + \ldots + \sum_{i=1}^{n} r_i \frac{\partial h_n}{\partial q_i} = 0,
\end{equation}
where $k_i$, $l_i$ and $r_i$ are some scalar functions. Soving  the system associated to equations (\ref{con 1}) and (\ref{con 2}) we  determine the functions $h_i$ and, consequently, the generators $v_i$ of our Lie algebra of symmetries.
As a consequence, we can state the following result:
\begin{proposition}
The infinitesimal symmetries of control system (\ref{eq2})-(\ref{eq4}) can be identified by calculating the flow associated with the vector field \( v_i \) at time \( s \), denoted by \( \gamma_i(s,t) = (q_1(s,t), \ldots, q_n(s,t)) \). This flow is obtained by solving the system of differential equations:
\[
\frac{\partial \gamma_i(s,t)}{\partial s} = v_i(\gamma_i(s,t)), \quad \gamma_i(0,t) = (q_1(t), \ldots, q_n(t)).
\]
\end{proposition}
\section{Infinitesimal Symmetries on SH(2) and the first Maxwell time}
\label{sec:4}
\subsection{Geometric control problem on $\text{SH}(2)$}
\label{sec:5}
The motion of a unicycle on a hyperbolic plane can be described using the following driftless control system
\begin{equation}
\label{eq6}
\left\{
\begin{aligned}
\dot{x} &= u_1 \cosh z, \\
\dot{y} &= u_1 \sinh z, \\
\dot{z} &= u_2,
\end{aligned}
\right.
\end{equation}
where $u_1$ is the translational velocity and $u_2$ is the angular velocity. The configuration and state manifold ${\rm M}$ of the system is three-dimensional, where, for any point $q = (x, y, z) \in {\rm M}$,   $x$ and $y$ are position variables and $z$ is the angular orientation variable of the unicycle on the hyperbolic plane. 
 A motion $m(x, y, z)$ on the configuration manifold ${\rm M}$, parameterized by $x, y, z \in \mathbb{R}$,  is a transformation that maps a point $\textbf{a}(a_1,a_2)$ to a point $\textbf{b}(b_1,b_2)$, such that:\\
\begin{equation}
\begin{aligned}
b_1 = a_1 \cosh z + a_2 \sinh z + x,\\
b_2 = a_1 \sinh z + a_2 \cosh z + y.
\end{aligned}
\end{equation}
Composition of two motions $m_1(x_1,y_1,z_1)$ and $m_2(x_2,y_2,z_2)$ is another motion $m_3(x_3,y_3,z_3)$ given as: 
$$
m_3(x_3,y_3,z_3)=m_1(x_1,y_1,z_1).m_2(x_2,y_2,z_3)
$$
where, 
\begin{equation}
\begin{aligned}
x_3 &=x_2\cosh z_1+y_2\sinh z_1+x_1,\\
y_3 &=x_2\sinh z_1+y_2\cosh z_1+y_1,\\
z_3 &=z_1+z_2
\end{aligned}
\end{equation}
The identity motion $m_{Id}$ is given by $x=y=z=0$, and inverse of a motion  $m(x,y,z)$ is given by $m^{-1}(x^{1},y^{1},z^{1})$ where, 
$$
\begin{aligned}
 x^{1}&=-x \cosh z+ y \sinh z,\\
 y^{1}&=x \sinh z- y \cosh z,\\
 z^{1}&=-z.
\end{aligned}
$$
The motions of the pseudo-Euclidean plane exhibit a group structure, with composition serving as the group operation. This group, known as the special hyperbolic group $\text{SH(2)}$, also possesses a smooth manifold structure, which qualifies it as a Lie group. One can equivalently formulate problem (\ref{eq6}) as a sub-Riemannian problem on $\text{SH}(2)$:
\begin{align}
& \dot{q}=u_1 X_1(q)+u_2 X_2(q), \quad q \in {\rm M}=\text{SH(2)}, \quad u=\left(u_1, u_2\right) \in \mathbb{R}^2, \label{eq7}\\
& q(0)=q_0=m_{Id}, \quad q\left(t_1\right)=q_1, \\
&J=\frac{1}{2} \int_0^{t_1}\left(u_1^2+u_2^2\right) d t \rightarrow \min
\end{align}
where,
\begin{equation}
X_1 = \cosh z \, \partial_{x} + \sinh z \, \partial_{y}, \quad X_2 = \partial_{z}.
\end{equation}
The problem under consideration is formulated on a sub-Riemannian manifold $(\mathrm{M}, \Delta, \mathrm{g})$. Here, $\mathrm{M}$ represents the underlying smooth manifold, $\Delta = \operatorname{span}\{X_1, X_2\}$ denotes a  distribution of rank 2, and $\mathrm{g}$ is a Riemannian metric defined on $\Delta$ such that $\rm{g}(X_i, X_j) = \delta_{ij}$, explicitly $g = dx^{2} - dy^{2} + dz^{2}$. The vector fields $X_1$ and $X_2$, which span the distribution $\Delta$, are left-invariant with respect to the group structure associated with $\mathrm{M}$. One can see that  the family  $\mathcal{F}=\lbrace u_1X_1+u_2X_2 \vert  u=(u_1, u_2) \in \mathbb{R}^2\rbrace$ is symmetric. Computing the vector field $X_3=[X_1,X_2]$, we get  $X_3=[X_1,X_2]=-\sinh z\partial_{x}-\cosh z\partial_{y}$.
 Since $X_1$, $X_2$, and $X_3$ are linearly independent, we have  ${\rm Lie}_q(\mathcal{F})={\rm span}(X_1(q),X_2(q),X_3(q))=T_q(SH(2))$ for every point $q$. This implies that  system (\ref{eq7}) is controllable. After transforming this system into an equivalent time-optimal problem with controls \( u_1^{2} + u_2^{2} \leqslant 1 \), one can observe that 
$$
\vert u_1 X_1 + u_2 X_2\rvert  \leqslant c \vert q \rvert, \quad q \in \rm{M}.
$$
This gives  the inequality $\vert q(t)\rvert \leqslant q_0 \exp(tc)$. Therefore, the attainable sets satisfy the following a priori bound:
$$
\mathcal{A}_{q_0}(\leq t_1) \subset \left\lbrace q \in \rm{M} \mid \vert q\rvert  \leqslant q_0 \exp(t_1c) \right\rbrace.
$$
Thus, according to Filipov's theorem (\cite{10}), optimal controls exist.
\subsection{Computing extremal trajectories}
We begin by applying the Pontryagin Maximum Principle \cite{1}, which allows  to determine the extremal trajectories of system  (\ref{eq6}). For this purpose, we define the functions $h_i(\lambda) = \langle \lambda, X_i \rangle \quad i=1,2,3$, where $\lambda \in T^*M$. These functions $(h_1, h_2, h_3)$, form a coordinate system on the fibers of $T^*M$. Consequently, we adopt the          global  coordinates $(q, h_1, h_2, h_3)$ to describe the structure of $T^*M$.
The Hamiltonian from the Pontryagin Maximum Principle (PMP) is given by:
\[
h_u^v(\lambda) = \frac{\nu}{2} \left(u_1^2 + u_2^2\right) + u_1 h_1(\lambda) + u_2 h_2(\lambda),
\]
where, $u = (u_1, u_2) \in \mathbb{R}^2$, and $v  \in \{-1, 0\}$.  
Then the Pontryagin maximum principle \cite{1} for the problem under consideration reads as follows.
\begin{theorem}(\cite{1})
Let $u(t)$ and $q(t), t \in\left[0, t_1\right]$, be an optimal control and the corresponding optimal trajectory in problem \ref{eq7}. Then there exist a Lipschitzian curve $\lambda_t \in T^* M, \pi\left(\lambda_t\right)=q(t), t \in\left[0, t_1\right]$, and a number $v \in\{-1,0\}$ for which the following conditions hold for almost all $t \in\left[0, t_1\right]$:
\begin{align}
 \dot{\lambda}_t&=\vec{h}_{u(t)}^v\left(\lambda_t\right), \\
 h_{u(t)}^\nu\left(\lambda_t\right) &=\max _{u \in \mathbb{R}^2} h_u^v\left(\lambda_t\right), \\
\left(v, \lambda_t\right) &\neq 0 \label{con 3}
\end{align}
\end{theorem}
Our analysis considers two distinct cases:\\
\textbf{- Abnormal case} \\
In this case $v=0$. Thus the Hamiltonian of PMP for the system takes the form:\\
$h_u(\lambda)= u_1h_1(\lambda)+u_2h_2(\lambda)$\\
$$
\begin{aligned}
\dot{h_1} &=\lbrace h_u(\lambda),h_1\rbrace\\
          &=u_1\lbrace h_1,h_1\rbrace+u_2\lbrace h_2,h_1\rbrace\\
          &=u_2h_3\\
\dot{h_2} &=-u_1h_3
\end{aligned}
$$
Then, Abnormal extremals satisfy the Hamiltonian system:

$$
\begin{aligned}
\dot{h_1}&=u_2h_3\\
\dot{h_2}&=-u_1h_3\\
\dot{q}&=u_1X_1+u_2X_2
\end{aligned}
$$
and the hypothesis $h_u(\lambda)\longrightarrow max$ implies  
\begin{center}
$h_1(\lambda_t)=h_2(\lambda_t)=0$
\end{center}
Therefore condition (\ref{con 3}) gives     $h_3(\lambda_t) \neq 0$  and the first two equations of the Hamiltonian system yield
$u_1(t)= u_2(t)=0$. So abnormal trajectories are constant.\\
\textbf{- Normal case}\\
The maximality condition\\
\begin{center}
$u_1h_1+u_2h_2- \frac{1}{2}(u_1^{2}+u_2^{2})\longrightarrow max$
\end{center}
yields $u_1=h_1$ and $u_2=h_2$,  and then then the Hamiltonnian is  $H=\frac{1}{2}(h_1^{2}+h_2^{2})$. Thus we get the system:\\
$$
\begin{aligned}
\dot{h_1}&=h_2h_3\\
\dot{h_2}&=-h_1h_3\\
\dot{h_3}&=h_1h_2\\
\dot{q}&=h_1X_1+h_2X_2\\
\end{aligned}
$$
and  Hamiltonian system in normal case is given by the equations:
\begin{align}
\dot{h}_1 &= h_2 h_3, \quad \dot{h}_2 = -h_1 h_3, \quad \dot{h}_3 = h_1 h_2  \label{vertical system}\\
\dot{x} &= h_1 \cosh z, \quad \dot{y} = h_1 \sinh z, \quad \dot{z} = h_2.
\end{align}
In this case, the initial covector $\lambda$ lies on the initial cylinder defined by  
\[ 
C = T_{q_0}^*M \cap \{H(\lambda) = \frac{1}{2}\}.
\]  
This set can be explicitly described as  
\[ 
C = \{(h_1, h_2, h_3) \in \mathbb{R}^3 \mid h_1^2 + h_2^2 = 1\}, 
\]  
representing a cylinder in the cotangent space where the Hamiltonian takes the constant value $\frac{1}{2}$.
We introduce the following change of variables:
\begin{align}
    h_1 = \cos(\alpha), \quad h_2 = \sin(\alpha).
\end{align}
With these coordinates, the vertical system (\ref{vertical system}) satisfies the equations:
\begin{align}
    \dot{h}_3 = \frac{1}{2} \sin(2\alpha), \quad \dot{\alpha} = h_3.
\end{align}
Next, we introduce another change of variables:
\begin{align}
    \gamma = 2\alpha \in \mathbb{R} / 4\pi \mathbb{Z}, \quad c = 2h_3 \in \mathbb{R}.
\end{align}
Finally, we obtain that the equation describing our vertical system corresponds to a mathematical pendulum given by:
\begin{align}
    \dot{\gamma} = c, \quad \dot{c} = -\sin(\gamma).
\end{align}
The total energy integral of the pendulum obtained is given as: $E=\frac{c^{2} }{2}- \cos(\gamma)= 2h_3^{2}-h_1^{2}+h_2^{2}$, according to this energy, the cylinder $C$ can be decomposed in the following way.
$C = \bigcup_{i=1}^{5} C_i$, where 
$$
\begin{aligned}
C_1 & =\{\lambda \in C \mid E \in(-1,1)\} \\
C_2 & =\{\lambda \in C \mid E \in(1, \infty)\} \\
C_3 & =\{\lambda \in C \mid E=1, c \neq 0\} \\
C_4 & =\{\lambda \in C \mid E=-1\}\\
C_5 & =\{\lambda \in C \mid E=1\}
\end{aligned}
$$
For a detailed derivation of the solutions to our system, refer to (\cite{11}).
\subsection{ Infinitesimal symmetries }
In this subsection, we compute the symmetry algebra of control  system (\ref{eq7}).
\begin{proposition}
Infinitesimal symmetries of the control system (\ref{eq7}) form a Lie algebra generated (over $\mathbb{R}$) by the vector fields:
\begin{equation}
\begin{aligned}
v_1&=-x\partial_{y}-y\partial_{x}-\partial_{z},\\
v_2&=\partial_{x},\\
v_3&=\partial_{y}.
\end{aligned}
\end{equation}
\end{proposition}
\begin{proof}
It is clear that $\Delta$ constitutes a contact distribution. Indeed, $\Delta = \ker(\omega)$, where  $\omega = \cosh(z)dy - \sinh(z) \in \Lambda^1(M)$ and we have $d\omega \wedge \omega \neq 0$.  Let  $v=f_1(x,y,z)X_1+f_2(x,y,z)X_2+f_3(x,y,z)X_3$ be a vector field.   The conditions $\omega(\mathcal{L}_{v}(X_1))=0$ and $\omega(\mathcal{L}_{v}(X_2))=0$, lead to the following system:
\begin{align}
f_2+\cosh z \frac{\partial f_3}{\partial x}+\sinh z \frac{\partial f_3}{\partial y} &= 0, \label{eq8}\\
f_1-\frac{\partial f_3}{\partial z} &=0. \label{eq9}
\end{align}
Moreover, the Lie derivative of the metric $\rm{g}$ has the form:\\
\begin{equation}
\begin{aligned}
\mathcal{L}_v(\rm{g}) = & 2(\cosh z \frac{\partial f_1}{\partial x}-\sinh z \frac{\partial f_3}{\partial x})dx^2 \\
& +2(\cosh z \frac{\partial f_3}{\partial y}-\sinh z \frac{\partial f_1}{\partial y})dy^2 \\
& +2(\frac{\partial f_2}{\partial z})dz^2 \\
& +(2\cosh z \frac{\partial f_1}{\partial y}-2\sinh z \frac{\partial f_3}{\partial y}-2\sinh z \frac{\partial f_1}{\partial x}+2\cosh z\frac{\partial f_3}{\partial x})dx dy \\
& + (2f_1 \sinh z+2\cosh z \frac{\partial f_1}{\partial z}-2f_3 \cosh z -2\sinh z \frac{\partial f_3}{\partial z}+2\frac{\partial f_2}{\partial x})dx dz \\
& + (-2\sinh z\frac{\partial f_1}{\partial z}-2f_1 \cosh z+2\cosh z\frac{\partial f_3}{\partial z} + 2f_3 \sinh z +2 \frac{\partial f_2}{\partial y})dz dy 
\end{aligned}
\end{equation}
Now the relation $\mathcal{L}_v(\rm{g}) = 0 $ implies:
\begin{align}
\cosh z \frac{\partial f_1}{\partial x}-\sinh z \frac{\partial f_3}{\partial x} &=0, \label{eq10}\\
\cosh z \frac{\partial f_3}{\partial y}-\sinh z \frac{\partial f_1}{\partial y} &=0 , \label{eq11}\\
\frac{\partial f_2}{\partial z}&=0,\\
2\cosh z \frac{\partial f_1}{\partial y}-2\sinh z \frac{\partial f_3}{\partial y}-2\sinh z \frac{\partial f_1}{\partial x}+2\cosh z\frac{\partial f_3}{\partial x}&=0,\\
2f_1 \sinh z+2\cosh z \frac{\partial f_1}{\partial z}-2f_3 \cosh z -2\sinh z \frac{\partial f_3}{\partial z}+2\frac{\partial f_2}{\partial x}&=0,\\
-2\sinh z\frac{\partial f_1}{\partial z}-2f_1 \cosh z+2\cosh z\frac{\partial f_3}{\partial z} + 2f_3 \sinh z +2 \frac{\partial f_2}{\partial y} &=0.
\end{align}
Using   equation (\ref{eq9}), we have $ f_1=\frac{\partial f_3}{\partial z}$. By substituting $f_1$ into equations (\ref{eq10}) and (\ref{eq11}), we obtain:
\begin{equation}
\begin{aligned}
\cosh z \frac{\partial^{2} f_3}{\partial z \partial x}-\sinh z \frac{\partial f_3}{\partial x}=0,\\
\cosh z \frac{\partial f_3}{\partial y}-\sinh z \frac{\partial^{2} f_3}{\partial z \partial y}=0.\\
\end{aligned}
\end{equation}
Thus 
\begin{equation}
\begin{aligned}
\frac{\partial f_3}{\partial x} &=c_1 \cosh z +f(x,y),\\
\frac{\partial f_3}{\partial y} &=c_2 \sinh z +g(x,y).
\end{aligned}
\end{equation}
By integrating equation (\ref{eq10}) and  (\ref{eq11}) with respect to $x$ and $y$, respectively, we obtain: \\
\begin{equation}
\begin{aligned}
\cosh z f_1-\sinh zf_3=0,\\
\cosh z f_3-\sinh z f_1=0.
\end{aligned}
\end{equation}
Then we get:\\
\begin{equation}
\begin{aligned}
\cosh z \frac{\partial f_3}{\partial z} -\sinh z f_3=0,\\
\cosh z f_3-\sinh z \frac{\partial f_3}{\partial z}=0.
\end{aligned}
\end{equation}
Therefore, we deduce that :
\begin{equation}
\begin{aligned}
f_3 &=c_1 x\cosh z + c_2 y\sinh z + c_3 \cosh z +c_4 \sinh z,\\
f_1 &=c_1 x\sinh z +c_2 y\cosh z+c_3 \sinh z +c_4 \cosh z. 
\end{aligned}
\end{equation} 
Using  equation (\ref{eq8}) we get  $f_2=-c_1$ and $c_2=-c_1$. Hence, we obtain:
\begin{equation}
\begin{aligned}
f_1 &=c_1 x\sinh z -c_1 y \cosh z +c_3 \sinh z +c_4 \cosh z\\
f_2 &=-c_1\\
f_3 &=c_1 x\cosh z -c_1 y\sinh z + c_3 \cosh z +c_4 \sinh z
\end{aligned}
\end{equation}      
For $c_1=1$, $c_3=c_4=0$, we find $v_1=-x\partial_{y}-y\partial_{x}-\partial_{z}$.
\end{proof}
\begin{proposition}
\label{pro action}
The action of the flow of the infinitesimal transformation $v_1=-x\partial_{y}-y\partial_{x}-\partial_{z}$,
at the time $s$ maps a geodesic with the initial condition $x(0) = y(0) =
z(0) = 0$ to another geodesic.
\begin{equation}
\label{eq12}
\begin{aligned}
x & \mapsto x \cosh(s) - y \sinh(s), \quad \quad \\
y & \mapsto y \cosh(s) - x \sinh(s), \\
z & \mapsto z - s.
\end{aligned}
\end{equation}
\end{proposition}
\subsection{First Maxwell time corresponding to infinitesimal symmetries}
The transformation (\ref{eq12}) can be viewed as the result of composing two distinct motions: the first motion $m(x,y,z)$ , which represents the geodesic under consideration, and a second motion $m(0,0,-s)$. Their combination results in a new motion $m(x',y',z')$, corresponding to a different geodesic, given by:
\begin{equation}
m(x',y',z')=m(0,0,-s)\cdot m(x,y,z)
\end{equation}
The motion $m(0,0,-s)$ is a transformation in the hyperbolic plane, mapping each point $(a_1, a_2)$ to another point $ (b_1, b_2)$, such that: 
\begin{equation}
\begin{aligned}
b_1& = a_1 \cosh s - a_2 \sinh s ,\\
b_2& = a_2 \cosh s - a_1 \sinh s.
\end{aligned}
\end{equation}
This motion  can be represented in matrix form as follows:
\[
N =
\begin{pmatrix}
\cosh(s) & -\sinh(s) & 0 \\
-\sinh(s) & \cosh(s) & 0 \\
0 & 0 & 1
\end{pmatrix}
\]
It is clear that the group generated by the matrix \( N \) is isomorphic to \( \rm{SO(1,2)} \), the Lorentz group. Therefore, Proposition (\ref{pro action}) yields the following result:
\begin{corollary}
For each $R \in \rm{SO}(1,2)$, the map
\begin{equation}
(x, y, z) \longmapsto R \cdot (x, y, z),
\end{equation}
which represents a motion in the group $\rm{M}$, maps geodesics starting at the origin to other geodesics. This map is defined by the composition of these two motions.
\end{corollary}
 It can be noted that the set of fixed points of this action is: 
\begin{equation}
\rm{S}=\lbrace (0,0,z)\quad  \vert z \in \mathbb{R}\rbrace
\end{equation}
\begin{proposition}
The intersection of our geodesics  with the set $\rm{S}$ provides the set of Maxwell points. Consequently, the following result is obtained:
\begin{enumerate}
\item  If $\lambda \in C_1 \cup C_3 \cup C_4$, the geodesics do  not intersect $\rm{S}$ for any $ t > 0$. 
\item If $\lambda \in C_2$, the set of Maxwell points is given by: $MAX=\lbrace (0,0,z(4nk_0K) \quad  k_0\in(0,1)\rbrace $ 
\item If $\lambda \in C_5 $, the set of Maxwell points is given by: $MAX=\lbrace (0,0,z(t)) \quad t>0\rbrace $
\end{enumerate}
\end{proposition}
\begin{proof}
\begin{enumerate}
\item  We begin with the case where $\lambda=(\varphi, k) \in C_1$, the geodesics intersect the set $\rm{S}$ if and only if: $x_t=y_t=0$. Let  $L_1=x_t+y_t$, and $L_2=x_t-y_t$, we then obtain the following expressions for $L_1$ and $L_2$ :
\begin{align}
    L_1 &= \omega \big( (E(\varphi_t) - E(\varphi)) - k (sn(\varphi_t) - sn(\varphi)) \big), \\
    L_2 &= \frac{1}{\omega(1 - k^2)} \big( (E(\varphi_t) - E(\varphi)) + k (sn(\varphi_t) - sn(\varphi)) \big).
\end{align}
The equations $L_1=0$ and  $L_2=0$ hold if and  anly if $(E(\varphi_t)-E(\varphi))=0$. This condition cannot be satisfied for all $t>0$. Therefore,  the geodesics do not intersect the set $\rm{S}$. In the case where  $\lambda \in  C_3 $, we have $L_1=\frac{1}{\omega}(\varphi_t-\varphi)$ and $L_2=2\omega(tanh(\varphi_t)-tanh(\varphi))$, since both equations cannot be obtained for each $t > 0$,  this leads to our result. In the last case, it is clear that  there is no intersection between our geodesic and $\rm{S}$.
\item If $\lambda \in C_2$. It is assumed that $L_1=x_t+y_t$ and $L_2=x_t-y_t$, which gives us the following expressions: 
\begin{align}
L_1 & = \omega \left( 
    -(E(\psi_t) - E(\psi)) + k'^2(\psi_t - \psi) + k \big(sn(\psi_t) - sn(\psi)\big)
\right),\\
L_2 &=\frac{1}{\omega(1-k^2)}\left( 
    (E(\psi_t) - E(\psi)) - k'^2 (\psi_t - \psi) + k \big(sn(\psi_t) - sn(\psi)\big)\right).
    \end{align}
We have $L_1=0$ and  $L_2=0$. if and only if both equations, $h_1=2k( \big(sn(\psi_t) - sn(\psi)\big)=0$ and \\ $h_2=(E(\psi_t) - E(\psi)) - k'^2 (\psi_t - \psi)=0$, hold. The equation $h_1 = 0$ holds exclusively when $t=4nkK$ with $k\in(0,1)$. Here $K(k)=\int_0^{\pi / 2} \frac{d t}{\sqrt{1-k^2 \sin ^2 t}}, \quad k'^2 =1-k^2$, for these given values, it follows that $h_2(4nkK)=4nE(k)-4nk'^2K$. We subsequently take $g(k)=E(k)-k'^2K$ $k\in \left[0, 1\right)$ we observe that its derivative $g'(k) = kK(k)$, indicating $g$ increases on the interval $\left[0, 1\right)$, with $g(\left[0, 1\right))=\left[0, 1\right)$. Therefore, there exists $k_0 \in \left(0, 1\right)$ such that  $g(k_0)=0$. Since $z(4nk_0K) = 0$, the set of Maxwell points is reduced to $\{(0, 0, 0)\}$.
\item If $\lambda \in C_5 $ For each $t$, we have $x_t = 0$ and $y_t = 0$, with $z_t \neq 0$. Subsequently, for each strictly positive $t$, our geodesic intersects $\rm{S}$. For more details on elliptic functions, (see \cite{3})
\end{enumerate}
\end{proof}
\begin{corollary}
The  first Maxwell time  $T_1^{Max}$ where our geodesic loses optimality corresponding to This action is given as:
$$
\begin{aligned}
\lambda \in C_1 \cup C_3\cup C_4 \Longrightarrow   T_1^{Max}=& +\infty \\
\lambda \in C_2 \Longrightarrow   T_1^{Max}= &4k_0K(k_0), \quad k_0 \in \left(0, 1\right) \\
\lambda \in  C_5 \Longrightarrow  T_1^{Max}=& t_0,\quad t_0\neq 0
\end{aligned}
$$
\end{corollary}
\begin{proof}
Whenever our geodesics do not intersect the set $\rm{S}$ they are optimal, and as a result $T_1^{Max}=+\infty$.
In case 2, we assign the $1$ to $n$ to find the first Maxwell time.
In case 5, there is always an intersection; consequently, our geodesic is no longer optimal. The first Maxwell time is a certain non-zero $t \neq 0$.
\end{proof}
\begin{figure}[h]
    \begin{minipage}{0.3\textwidth}
       \centering
        \includegraphics[scale=0.3]{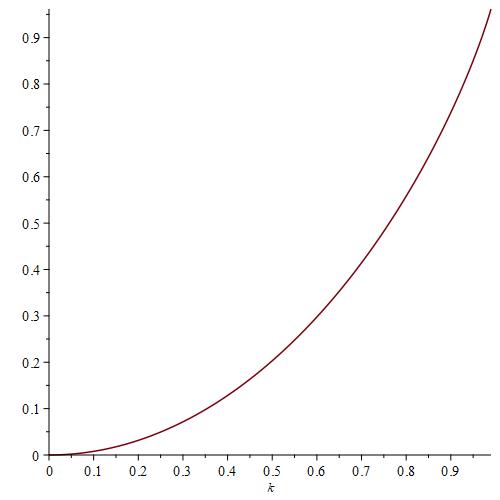}  
       \caption{The graph of the function $g$}
        \label{Figure 1}
    \end{minipage}
\end{figure}
Here, we include several figures to conclude this note and visually demonstrate our method. We use the notation $x'$, $y'$, and $z'$ to refer to the symmetries of our trajectories.
\begin{figure}[h]
    \centering
    \begin{minipage}{0.3\textwidth}
       \centering
      \includegraphics[scale=0.3]{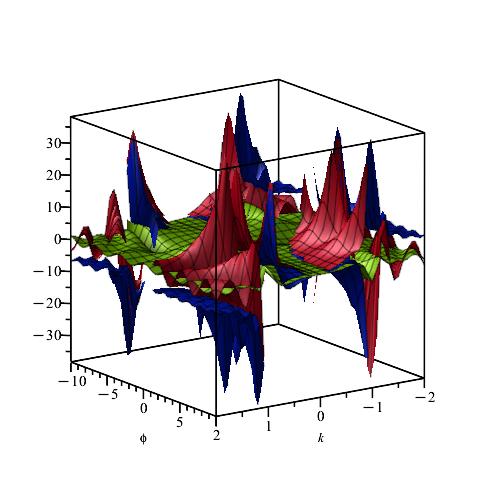}  
       \caption{Local minimizer, in the case $\lambda=(\varphi, k) \in C_1$}
        \label{Figure 2}
   \end{minipage}\hfill
    \begin{minipage}{0.3\textwidth}
        \centering
        \includegraphics[scale=0.3]{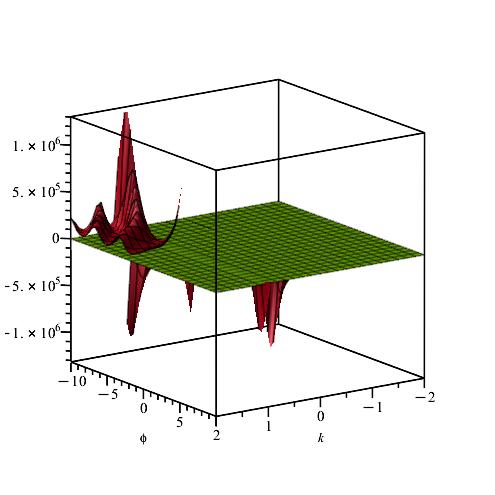}
        \caption{The trajectory's symmetric}
        \label{Figure 3}
    \end{minipage}\hfill
    \begin{minipage}{0.3\textwidth}
        \centering
        \includegraphics[scale=0.3]{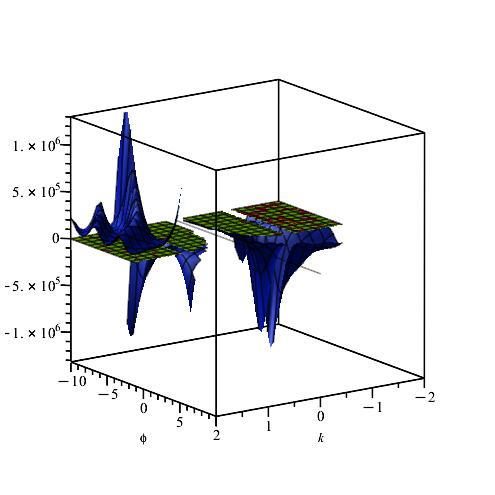}
        \caption{The trajectory $(x(t), y(t))$ and its symmetry, in the case $\lambda \in C_1$}
        \label{Figure 4}
    \end{minipage}
\end{figure}
\begin{figure}[h]
    \centering
    \begin{minipage}{0.3\textwidth}
        \centering
        \includegraphics[scale=0.3]{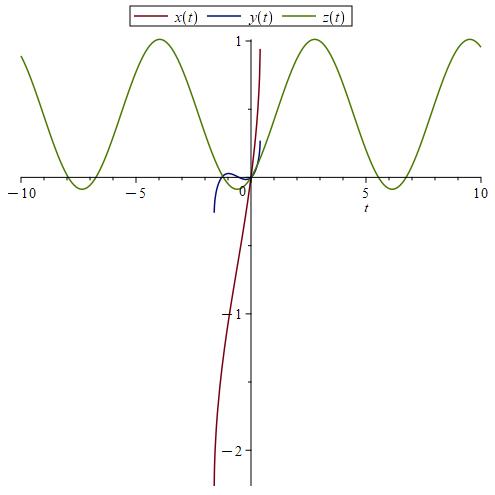}
        \caption{The local minimizer for numerical values, where $\lambda=(\varphi, k) \in C_1$.}
        \label{Figure 5}
    \end{minipage}\hfill
    \begin{minipage}{0.3\textwidth}
        \centering
         \includegraphics[scale=0.3]{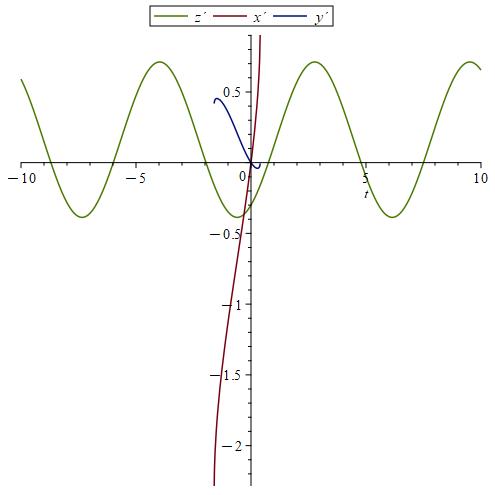}
        \caption{The trajectory's symmetric for numerical values}
        \label{Figure 6}
    \end{minipage}\hfill
    \begin{minipage}{0.3\textwidth}
        \centering
        \includegraphics[scale=0.3]{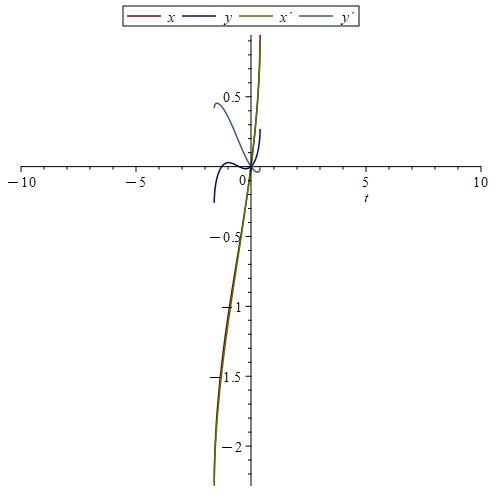}
        \caption{The trajectory $(x(t), y(t))$  and its symmetry for numerical values, where $\lambda\in C_1$.}
        \label{Figure 7}
    \end{minipage}
\end{figure}
\begin{figure}[h]
    \centering
   \begin{minipage}{0.3\textwidth}
        \centering
        \includegraphics[scale=0.3]{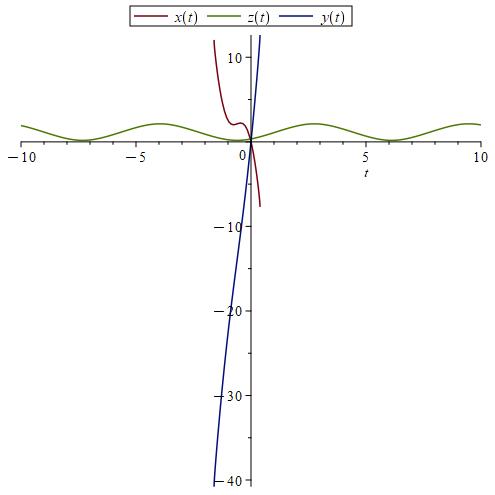}   
        \caption{The local minimizer for numerical values, where $\lambda=(\varphi, k) \in C_2$.}
    \end{minipage}\hfill
    \begin{minipage}{0.3\textwidth}
        \centering
       \includegraphics[scale=0.3]{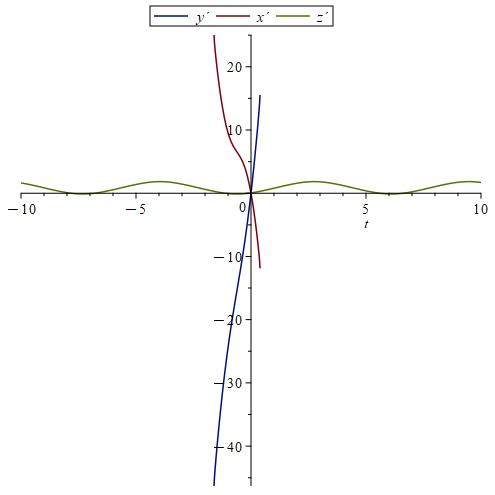}  
       \caption{The trajectory's symmetric for numerical values}
    \end{minipage}\hfill
    \begin{minipage}{0.3\textwidth}
        \centering
        \includegraphics[scale=0.3]{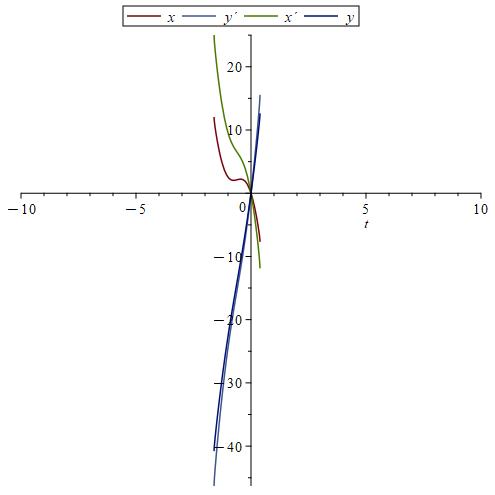}
        \caption{The trajectory $(x(t),y(t))$ and its symmetry for numerical value, where $\lambda \in C_2$}
        \label{Figure 10}
    \end{minipage}
\end{figure}
\newpage
Upon analyzing the preceding figures, several key observations emerge. Figures \ref{Figure 2}, \ref{Figure 3}, and \ref{Figure 4} provide a numerical-free overview of our trajectory and its symmetry, showcasing their intersections. Moving to Figure \ref{Figure 5}, we delve into a specific case where numerical values are assigned to $\varphi$ and $k$, allowing us to visualize the trajectories $x(t)$, $y(t)$, and $z(t)$. Figure \ref{Figure 6} mirrors this setup, illustrating the symmetry of our trajectory using the same numerical values. Transitioning to the $(x, y)$ plane, Figure \ref{Figure 7} displays the intersection between our trajectory and its symmetry, highlighting the absence of intersections for strictly positive times. This analytical approach is replicated for the second case, culminating in Figure \ref{Figure 10}, which reveals an intersection at a point closer to 0, with $k_0$ minimized.

\end{document}